\newif\ifJOURNAL
\JOURNALfalse
\newif\ifWP
\WPfalse
\newif\ifBASIC
\BASICfalse

\newif\ifFULL
\FULLfalse
\newif\ifLATIN
\LATINfalse

\WPtrue
\LATINtrue		% LATIN means that the Cyrillic references should be set in Latin
\newif\ifnotJOURNAL	% derivative conditional
\notJOURNALtrue
\ifJOURNAL\notJOURNALfalse\fi

\newif\ifnotWP		% derivative conditional
\notWPtrue
\ifWP\notWPfalse\fi

\newif\ifnotLATIN	% derivative conditional
\notLATINtrue
\ifLATIN\notLATINfalse\fi

\ifnotWP
  \newcommand{\ShaferLevy}{shafer/etal:arXiv0905}
  \newcommand{\ShaferVovk}{shafer/vovk:2006SS}
\fi
\ifWP
  \newcommand{\ShaferLevy}{GTP27}
  \newcommand{\ShaferVovk}{GTP4}
\fi

\ifnotLATIN

\fi
\ifLATIN

\fi

\ifJOURNAL
\documentclass[aos,preprint]{imsart}

\RequirePackage[OT1]{fontenc}
\RequirePackage{amsthm,amsmath,amsfonts,natbib}

\setcitestyle{square}		% added by me

\arxiv{stat.TH/0000000}

\startlocaldefs
\numberwithin{equation}{section}
\theoremstyle{plain}

\endlocaldefs
\fi

\ifWP
  \documentclass{article}
  \usepackage{amsmath,amsthm,amsfonts,amssymb,latexsym,graphicx}
\makeatletter

\newif\iftwodates
\twodatesfalse

\renewcommand\maketitle{\begin{titlepage}%
  \let\footnotesize\small
  \let\footnoterule\relax
  \let \footnote \thanks
  \null\vfil
  \vskip 30\p@
  \begin{center}%
    {\LARGE \bf \@title \par}%
    \vskip 3em%
    {\large
     \lineskip .75em%
     \begin{tabular}[t]{c}%
       \@author
     \end{tabular}\par}%
     \vskip 1.5em%
  \end{center}\par
  \vfill
  \begin{center}
    \raisebox{1.5cm}{\includegraphics[width=0.58\textwidth]{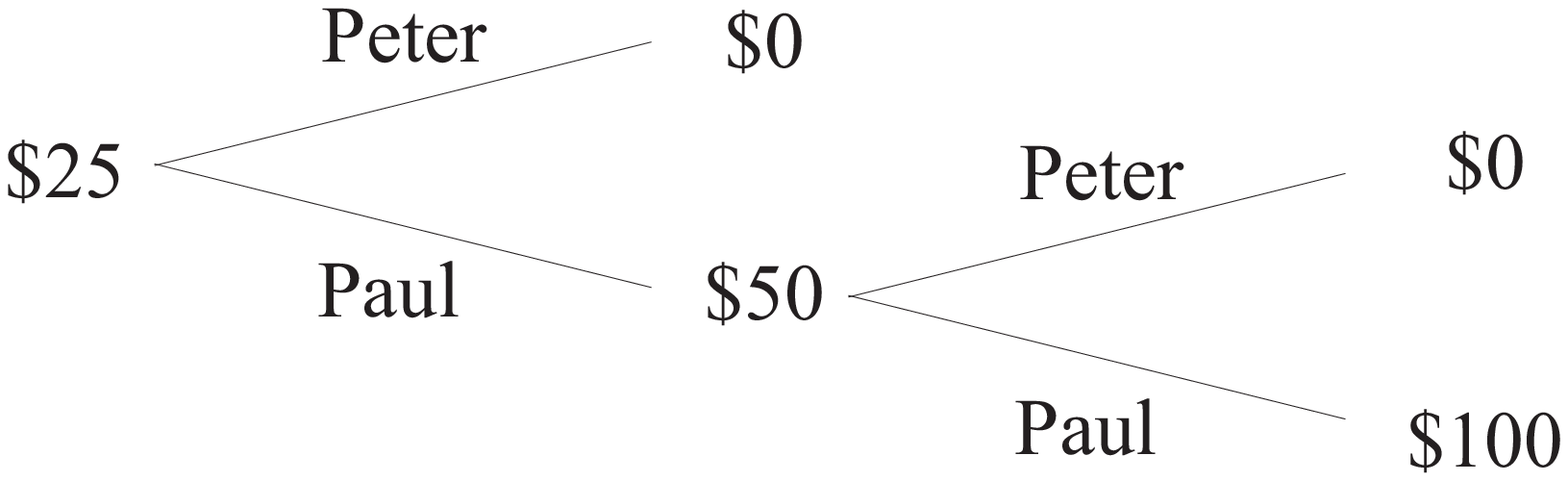}}%
    \hskip 3em%
    \includegraphics[width=0.29\textwidth]{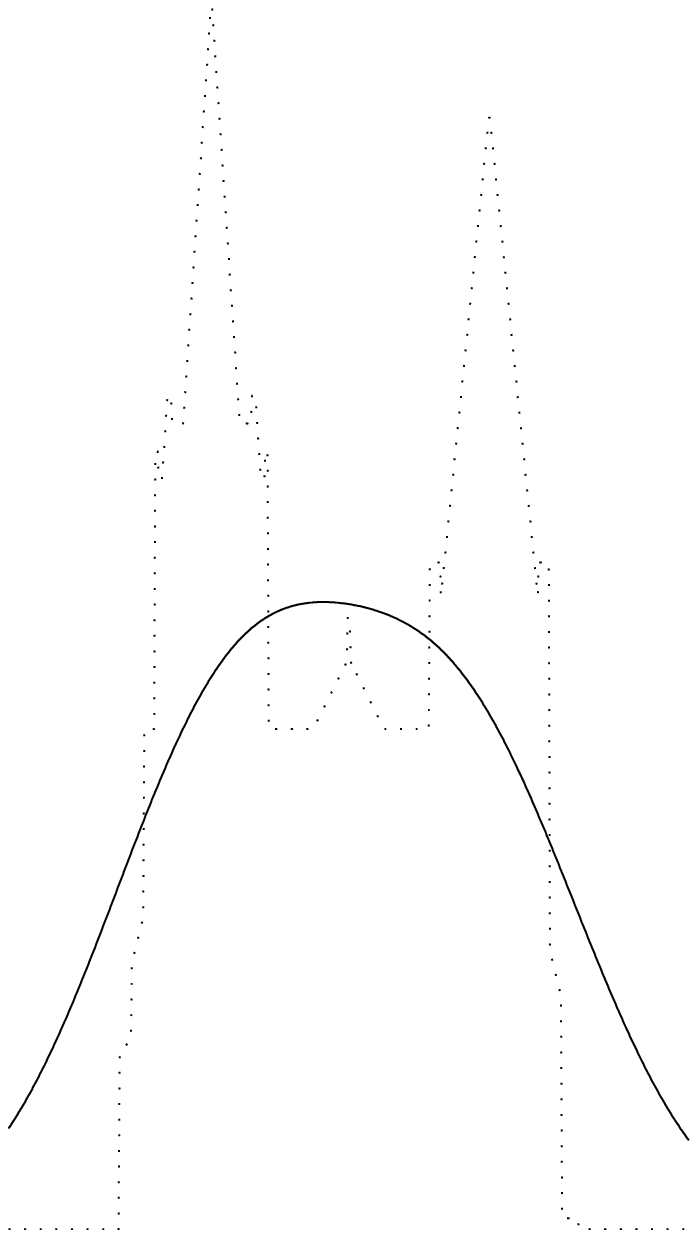}%
  \end{center}
  \@thanks
  \vfill
  \begin{center}
    {\large \bf The Game-Theoretic Probability and Finance Project}
  \end{center}
  \begin{center}
    {\large Working Paper \#\No}
  \end{center}
  \begin{center}
    {\iftwodates\large First posted \firstposted.
    Last revised \@date.\else\large\@date\fi}
  \end{center}
  \begin{center}
    Project web site:\\
    http://www.probabilityandfinance.com
  \end{center}
  \end{titlepage}%
  \setcounter{footnote}{0}%
  \global\let\thanks\relax
  \global\let\maketitle\relax
  \global\let\@thanks\@empty
  \global\let\@author\@empty
  \global\let\@date\@empty
  \global\let\@title\@empty
  \global\let\title\relax
  \global\let\author\relax
  \global\let\date\relax
  \global\let\and\relax
}

\renewenvironment{abstract}{%
  \titlepage
  \null\vfil
  \@beginparpenalty\@lowpenalty
  \begin{center}%
    \Large \bfseries \abstractname
    \@endparpenalty\@M
  \end{center}}%
  {\par\vfill\tableofcontents\endtitlepage}

\renewenvironment{thebibliography}[1]
  {\section*{\refname}%
  \addcontentsline{toc}{section}{\refname}%   Volodya
  \@mkboth{\MakeUppercase\refname}{\MakeUppercase\refname}%
  \list{\@biblabel{\@arabic\c@enumiv}}%
    {\settowidth\labelwidth{\@biblabel{#1}}%
    \leftmargin\labelwidth
    \advance\leftmargin\labelsep
    \@openbib@code
    \usecounter{enumiv}%
    \let\p@enumiv\@empty
    \renewcommand\theenumiv{\@arabic\c@enumiv}}%
    \sloppy
    \clubpenalty4000
    \@clubpenalty \clubpenalty
    \widowpenalty4000%
    \sfcode`\.\@m}
    {\def\@noitemerr
    {\@latex@warning{Empty `thebibliography' environment}}%
  \endlist}

\makeatother

\fi

\ifBASIC
  \documentclass[12pt]{article}
  \usepackage{amsmath,amsthm,amsfonts,amssymb,latexsym}
\fi

\ifFULL
  \usepackage{color}

  \newcommand{\bluebegin}{\begingroup\color{blue}}
  \newcommand{\blueend}{\endgroup}

\fi

\allowdisplaybreaks[4]

\newcommand{\Vladimir}{Vladimir}
\newcommand{\DOT}{.}

\ifnotLATIN
  \input{OT2enc.def}

\fi

\newcommand{\st}{\mathrel{|}}
\newcommand{\givn}{\mathrel{|}}

\newcommand{\bbbr}{\mathbb{R}}		% the real numbers
\newcommand{\bbbn}{\mathbb{N}}		% the natural numbers
\newcommand{\bbbp}{\mathbb{P}}		% auxiliary (probability)
\newcommand{\bbbe}{\mathbb{E}}		% auxiliary (expectation)

\DeclareMathOperator{\Prob}{\bbbp}			% probability
\DeclareMathOperator{\UpProb}{\bbbp^{{\rm game}}}	% upper game-theoretic probability
\DeclareMathOperator{\LowProb}{\underline{\bbbp}^{{\rm game}}}
\DeclareMathOperator{\UpProbMeas}{\bbbp^{{\rm meas}}}	% upper measure-theoretic probability 
\DeclareMathOperator{\LowProbMeas}{\underline{\bbbp}^{{\rm meas}}}
\DeclareMathOperator{\UpExpect}{\bbbe^{{\rm game}}}	% upper game-theoretic expectation

\DeclareMathOperator{\UpExpectMeas}{\bbbe^{{\rm meas}}}	% upper measure-theoretic expectation

\newcommand{\FFF}{\mathcal{F}}		% the main class of prequential events

\theoremstyle{plain}
\newtheorem{theorem}{Theorem}
\newtheorem*{ChoquetTheorem}{Choquet's Capacitability Theorem}
\newtheorem{proposition}{Proposition}

\newtheorem{lemma}{Lemma}

\theoremstyle{definition}

\newtheorem{remark}{Remark}

\newlength{\IndentI}
\newlength{\IndentII}
\newlength{\IndentIII}
\newlength{\IndentIV}
\newlength{\WidthI}
\newlength{\WidthII}
\newlength{\WidthIII}
\newlength{\WidthIV}
\ifWP
  \title{Prequential probability:\\game-theoretic = measure theoretic}
  \author{Vladimir Vovk}
  \newcommand{\No}{27}
  \twodatestrue
  \newcommand{\firstposted}{January 26, 2009}
\fi

\ifBASIC
  \title{Prequential probability:\\game-theoretic = measure theoretic}
  \author{Vladimir Vovk\\
  \texttt{http://vovk.net}}
\fi

\begin{document}
\ifJOURNAL
\begin{frontmatter}
  \title{Prequential probability:\protect\\
    game-theoretic = measure theoretic\protect\thanksref{T}}
  \runtitle{Prequential probability}
\thankstext{T}{This work was supported in part by EPSRC (grant EP/F002998/1).}

\begin{aug}
  \author{\fnms{Vladimir} \snm{Vovk}
    \ead[label=e]{vovk@cs.rhul.ac.uk}
    \ead[label=u,url]{http://vovk.net}}

\runauthor{Vladimir Vovk}

\affiliation{Royal Holloway, University of London}

\address{Vladimir Vovk\\
  Department of Computer Science\\
  Royal Holloway, University of London\\
  Egham, Surrey TW20 0EX, UK\\
  \printead{e}\\
  \printead{u}}
\end{aug}
\fi

\ifnotJOURNAL
  \maketitle
\fi

\begin{abstract}
  This article continues study of the prequential framework
  for evaluating a probability forecaster.
  Testing the hypothesis that the sequence of forecasts issued by the forecaster
  is in agreement with the observed outcomes
  can be done using prequential notions of probability.
  It turns out that there are two natural notions of probability
  in the prequential framework:
  game-theoretic, whose idea goes back to von Mises and Ville,
  and measure-theoretic, whose idea goes back to Kolmogorov.
  The main result of this article is that,
  in the case of predicting binary outcomes,
  the two notions of probability in fact coincide
  on the analytic sets
  (in particular, on the Borel sets).
\end{abstract}

\ifJOURNAL

\end{frontmatter}
\fi

\section{Background}
\label{sec:background}

The prequential framework for evaluating probability forecasters
was introduced by A.~P.~Dawid in \cite{dawid:1984} and \cite{dawid:1985}.
Suppose two players, Forecaster and Reality,
interact according to the following protocol.

\bigskip

\noindent
\textsc{Binary prequential protocol}\nopagebreak

\smallskip

\parshape=4
\IndentI  \WidthI
\IndentII \WidthII
\IndentII \WidthII
\IndentI  \WidthI
\noindent
FOR $n=1,2,\dots$:\\
  Forecaster announces $p_n\in[0,1]$.\\
  Reality announces $y_n\in\{0,1\}$.\\
END FOR.

\bigskip

\noindent
The interpretation is that $p_n$ is Forecaster's subjective probability that $y_n=1$
after having observed $y_1,\ldots,y_{n-1}$
and taking account of all other relevant information
available at the time of issuing the forecast.
We will refer to $p_n$ as \emph{forecasts} and to $y_n$ as \emph{outcomes}.
More generally, the outcomes take values
in an arbitrary measurable space
and the forecasts are probability distributions on that measurable space,
but in this article we will restrict our attention to binary outcomes
(as in \cite{dawid:1985});
this will be further discussed at the end of Section \ref{sec:application}.

In general,
the two players possess perfect information
about each other's moves:
Forecaster chooses $p_1$,
Reality observes $p_1$ and chooses $y_1$,
Forecaster observes $y_1$ and chooses $p_1$,
etc.
We might, however, be interested in ``oblivious'' strategies for a player,
especially for Reality,
who may generate her moves randomly
according to a probability measure on $\{0,1\}^{\infty}$
chosen in advance.
On the other hand,
the players may also react to events outside the protocol.

Dawid's \emph{prequential principle}
(see, e.g., \cite{dawid:1984,dawid:1985,dawid/vovk:1999})
says that when testing the adequacy of the forecaster
in light of the outcomes $y_n$
we should only use the forecasts $p_n$,
not the forecasting strategy (if any)
that Forecaster used to produce $p_n$.
In this article we will be only interested in testing procedures
that respect the prequential principle.
In other words,
we will be interested in testing the sequence
\begin{equation}\label{eq:sequence}
  (p_1,y_1,p_2,y_2,\ldots)
\end{equation}
of forecast/outcome pairs $(p_n,y_n)$ for agreement.
This sequence may be infinite or finite.

There are two main ways to test sequences (\ref{eq:sequence})
for agreement,
which we will call game-theoretic and measure-theoretic.
For concreteness,
suppose the sequence (\ref{eq:sequence}) does not satisfy
\begin{equation}\label{eq:calibration}
  \lim_{n\to\infty}
  \frac{1}{n}
  \sum_{i=1}^n
  (y_i-p_i)
  =
  0
\end{equation}
(i.e., the sequence is not ``unbiased in the large'';
see, e.g., \cite{dawid/vovk:1999}
for numerous other ways of testing probability forecasts).
What do we mean when we say that violation of (\ref{eq:calibration})
evidences lack of agreement?

Two ways to answer this question
correspond to two different approaches to the foundations of probability theory.
One version of the game-theoretic answer is that we can gamble against the forecasts
is such a way that, risking only one monetary unit,
we can become infinitely rich
when (\ref{eq:calibration}) is violated.
The measure-theoretic answer is that,
no matter what strategy Forecaster is using,
the probability of (\ref{eq:calibration}) is one;
therefore,
if (\ref{eq:calibration}) is violated,
an \emph{a priori} specified event of probability zero
(given by the negation of the formula (\ref{eq:calibration}))
has occurred.
Both becoming infinitely rich
and the occurrence of a pre-specified event of probability zero
can be interpreted as lack of agreement between the forecasts and outcomes.

In fact,
even the first answer can be expressed in terms of probability.
The game-theoretic approach to the foundations of probability
is as old as the standard measure-theoretic
based on Kolmogorov's axioms
(\cite{kolmogorov:1933};
see \cite{\ShaferVovk} for the historical background).
An imperfect version of the game-theoretic approach
was championed by von Mises \cite{mises:1919}
and formalized, in different ways,
by Wald \cite{wald:1937} and Church \cite{church:1940}.
Ville \cite{ville:1939} gave an example demonstrating
that von Mises's notion of a gambling strategy was too restrictive,
and introduced a more general class of gambling strategies
and a closely related notion of a martingale.
However, the formal notion of game-theoretic probability
was introduced only recently
(see, e.g., \cite{vovk:1993logic}, \cite{dawid/vovk:1999},
or, for a much fuller treatment, \cite{shafer/vovk:2001}).
In particular,
an event has zero game-theoretic probability
if and only if there is a gambling strategy that,
risking at most one monetary unit,
makes the player infinitely rich when the event happens.

The notion of game-theoretic probability makes
the game-theoretic and measure-theoretic justifications
of the testing procedure based on (\ref{eq:calibration})
look very similar:
we just say that the probability
(either game-theoretic or measure-theoretic)
of (\ref{eq:calibration}) being violated is zero.
The main result of this article says
that the two notions of probability coincide on the analytic sets,
and so the two approaches to testing probability forecasts are equivalent,
in the prequential framework.
The restriction to the analytic, and even Borel, sets is not a limitation
in all practically interesting cases.

For testing procedures based on events of probability zero
(basically, on strong laws of probability theory,
such as (\ref{eq:calibration})),
a special case of our result is sufficient:
it is sufficient to know that a Borel set
has zero game-theoretic probability
if and only if
it has zero measure-theoretic probability.
Our full result is also applicable to events of merely low,
not zero,
probability.
For example, we could reject the hypothesis of agreement if
\begin{equation}\label{eq:finite-calibration}
  \frac1n
  \left|
    \sum_{i=1}^n
    (y_i-p_i)
  \right|
  \ge
  C\sqrt{n}
\end{equation}
for prespecified large numbers $C$ and $n$.
% with $C$ large
Our result shows that this and similar procedures
have equally strong game-theoretic and measure-theoretic justifications.
Notice that in the case of (\ref{eq:finite-calibration})
our decision to reject the hypothesis of agreement
can be made after observing a finite sequence,
$(p_1,y_1,\ldots,p_n,y_n)$.

\section{This article}

In the following two sections,
\ref{sec:game} and \ref{sec:measure},
we formally introduce in the prequential framework
the two notions of probability discussed in the previous section.
The main result of this article is Theorem \ref{thm:main}
in Section \ref{sec:result},
asserting the coincidence of the two kinds of probability
on the analytic sets.
This result has several predecessors.
% which we state in our framework
In the situation where Forecaster's strategy is fixed,
Ville (\cite{ville:1939}, Theorems 1 and 2 in Chapter IV)
showed that a set $E$ has game-theoretic probability zero
if and only if it has measure-theoretic probability zero.
(Ville stated this result in slightly different terms,
without explicit use of game-theoretic probability.)
This was generalized in \cite{shafer/vovk:2001}
(Proposition 8.13)
to the statement that game-theoretic and measure-theoretic probability
coincide on the Borel sets.
In the case of a finite-horizon protocol,
a statement analogous to Theorem \ref{thm:main}
was proved by Shafer in \cite{shafer:1996art} (Proposition 12.7.4).
The special case of Theorem \ref{thm:main}
asserting the coincidence of game-theoretic and measure-theoretic probability
on the open sets
was first proved in \cite{vovk/shen:2008journal} (Theorem 2).

In the same Section \ref{sec:result}
we also prove that measure-theoretic probability never exceeds
game-theoretic probability.
This simple statement is true for all sets,
not just analytic.
The proof of the opposite inequality is given in Section \ref{sec:proof}.
It relies on two fundamental results:
Choquet's capacitability theorem \cite{choquet:1954}
and L\'evy's zero-one law in its game-theoretic version
recently found in \cite{\ShaferLevy}.

\subsection*{Some notation and definitions}

The set of all natural (i.e., positive integer) numbers is denoted $\bbbn$,
$\bbbn:=\{1,2,\ldots\}$.
% $\overline{\bbbn}_0$ is $\bbbn$ extended by adding $\infty$ and $0$
As always, $\bbbr$ is the set of all real numbers.

Let $\Omega:=\{0,1\}^{\infty}$
be the set of all infinite binary sequences
and $\Omega^{\diamond}:=\cup_{n=0}^{\infty}\{0,1\}^n$
be the set of all finite binary sequences.
Set $\Pi:=([0,1]\times\{0,1\})^{\infty}$
and $\Pi^{\diamond}:=\cup_{n=0}^{\infty}([0,1]\times\{0,1\})^n$.
The empty element (sequence of length zero)
of both $\Omega^{\diamond}$ and $\Pi^{\diamond}$
will be denoted $\Lambda$.
In our applications,
the elements of $\Omega$ and $\Omega^{\diamond}$ will be sequences of outcomes
(infinite or finite),
and the elements of $\Pi$ and $\Pi^{\diamond}$
will be sequences of forecasts and outcomes (infinite or finite).
The set $\Pi$ will sometimes be referred to as the \emph{prequential space}.

For $x\in\Omega^{\diamond}$,
let $\Gamma(x)\subseteq\Omega$ be the set of all infinite extensions of $x$
that belong to $\Omega$.
Similarly, for $x\in\Pi^{\diamond}$,
$\Gamma(x)\subseteq\Pi$
is the set of all infinite extensions of $x$
that belong to $\Pi$.
For each $\omega=(y_1,y_2,\ldots)\in\Omega$ and $n\in\bbbn\cup\{0\}$,
set $\omega^n:=(y_1,\ldots,y_n)$.
Similarly, for each $\pi=(p_1,y_1,p_2,y_2,\ldots)\in\Pi$ and $n\in\bbbn\cup\{0\}$,
set $\pi^n:=(p_1,y_1,\ldots,p_n,y_n)$.
% we will also be using the notation
% $\pi^{n-}:=(p_1,y_1,\ldots,p_n,y_n)$

In some proofs and remarks
we will be using the following notation, for $n\in\bbbn\cup\{0\}$:
$\Omega^n:=\{0,1\}^n$ is the set of all finite binary sequences of length $n$;
$\Omega^{\ge n}:=\cup_{i=n}^{\infty}\Omega^i$
% (resp.\ $\Omega^{\le n}$, $\Omega^{<n}$, $\Omega^{>n}$)
is the set of all finite binary sequences of length at least
% (resp.\ at most, less than, more than)
$n$;
$\Pi^n:=([0,1]\times\{0,1\})^n$;
$\Pi^{\ge n}:=\cup_{i=n}^{\infty}([0,1]\times\{0,1\})^i$.

\section{Game-theoretic prequential probability}
\label{sec:game}

A \emph{farthingale} is a function $V:\Pi^{\diamond}\to(-\infty,\infty]$ satisfying
\begin{multline}\label{eq:farthingale}
   V(p_1,y_1,\ldots,p_{n-1},y_{n-1})\\
   =
   (1-p_n)
   V(p_1,y_1,\ldots,p_{n-1},y_{n-1},p_n,0)\\
   +
   p_n
   V(p_1,y_1,\ldots,p_{n-1},y_{n-1},p_n,1)
\end{multline}
for all $n\in\bbbn$ and all $(p_1,y_1,p_2,y_2,\ldots)\in\Pi$;
the product $0\infty$ is defined to be $0$.
If we replace ``$=$'' by ``$\ge$'' in (\ref{eq:farthingale}),
we get the definition of a \emph{superfarthingale}.
These are prequential versions of the standard notions
of martingale and supermartingale.
We will be interested mainly in non-negative farthingales and superfarthingales.

The value of a farthingale can be interpreted
as the capital of a gambler
betting according to the odds announced by Forecaster.
In the case of superfarthingales,
the gambler is allowed to throw away part of his capital.

Game-theoretic probability can be introduced as either upper or lower probability;
in this article the former is more convenient
(and was used in the informal discussion of Section \ref{sec:background}).
A \emph{prequential event} is a subset of $\Pi$.
The \emph{upper game-theoretic probability} of a prequential event $E$ is
\begin{equation}\label{eq:upper-game}
  \UpProb(E)
  :=
  \inf
  \left\{
    a
    \st
    \exists V:
    V(\Lambda)=a
    \text{ and }
    \forall\pi\in E:
    \limsup_n V(\pi^n)\ge1
  \right\},
\end{equation}
where $V$ ranges over the non-negative farthingales.
It is clear that we will obtain the same notion of upper game-theoretic probability
if we replace the $\ge$ in (\ref{eq:upper-game}) by $>$,
replace $\limsup$ by $\sup$ or $\liminf$
(we can always stop when 1 is reached),
or allow $V$ to range over the non-negative superfarthingales.

We will need the following property of countable subadditivity
of game-theoretic probability.

\begin{lemma}\label{lem:subadditivity}
  For any sequence $E_1,E_2,\ldots$ of prequential events,
  \begin{equation*}
    \UpProb
    \left(
      \cup_{i=1}^{\infty}
      E_i
    \right)
    \le
    \sum_{i=1}^{\infty}
    \UpProb(E_i).
  \end{equation*}
  In particular,
  if $\UpProb(E_i)=0$ for all $i$,
  then $\UpProb(\cup_{i=1}^{\infty}E_i)=0$.
\end{lemma}
\begin{proof}
  It suffices to notice that the sum of a sequence of non-negative farthingales
  is again a non-negative farthingale.
\end{proof}

\section{Measure-theoretic prequential probability}
\label{sec:measure}

A \emph{forecasting system} is a function $\phi:\Omega^{\diamond}\to[0,1]$.
Let $\Phi$ be the set of all forecasting systems.
For each $\phi\in\Phi$
there exists a unique probability measure $\Prob_{\phi}$
on $\Omega$ (equipped with the Borel $\sigma$-algebra)
such that, for each $x\in\Omega^{\diamond}$,
$\Prob_{\phi}(\Gamma(x1)) = \phi(x)\Prob_{\phi}(\Gamma(x))$.
(In other words,
such that $\phi(x)$ is a version of the conditional probability,
according to $\Prob_{\phi}$,
that $x$ will be followed by $1$.)
The notion of a forecasting system is close
to that of a probability measure on $\Omega$:
the correspondence $\phi\mapsto\Prob_{\phi}$
becomes an isomorphism
if we only consider forecasting systems
taking values in the open interval $(0,1)$
and probability measures taking positive values
on the sets $\Gamma(x)$, $x\in\Omega^{\diamond}$.

For each sequence $(y_1,\ldots,y_n)\in\Omega^{\diamond}$
and each forecasting system $\phi\in\Phi$,
let
\begin{equation*}
 (y_1,\ldots,y_n)^{\phi}
 :=
 (\phi(\Lambda),y_1,\phi(y_1),y_2,\ldots,\phi(y_1,\ldots,y_{n-1}),y_n)
 \in
 \Pi^{\diamond}.
\end{equation*}
Similarly,
for each $(y_1,y_2,\ldots)\in\Omega$ and each $\phi\in\Phi$,
\begin{equation*}
 (y_1,y_2,\ldots)^{\phi}
 :=
 (\phi(\Lambda),y_1,\phi(y_1),y_2,\phi(y_1,y_2),y_3,\ldots)
 \in
 \Pi.
\end{equation*}

We can apply the idea of measure-theoretic probability to prequential events as follows,
in the spirit of \cite{huber:1981}, Section 10.2.
For each forecasting system $\phi$ and prequential event $E\subseteq\Pi$,
define
\begin{equation*}%\label{eq:P-phi}
  \Prob^{\phi}(E)
  :=
  \Prob_{\phi}
  \left\{
    \omega\in\Omega
    \st
    \omega^{\phi}\in E
  \right\}
  =
  \Prob_{\phi}(E^{\phi}),
\end{equation*}
where
$
  E^{\phi}
  :=
  \left\{
    \omega\in\Omega
    \st
    \omega^{\phi}\in E
  \right\}
$
and $\Prob_{\phi}(A)$ is understood, in general, as the outer measure of $A$,
i.e., as $\inf_{B}\Prob_{\phi}(B)$,
$B$ ranging over the Borel sets containing $A$.
The convention about using the outer measure is important only for our proofs,
not for the statement of the main result:
according to Luzin's theorem
(see, e.g., \cite{kechris:1995}, Theorem 21.10),
every analytic set is universally measurable,
and $E^{\phi}$ is analytic whenever $E$ is.
% in general, we understand the $\Prob_{\phi}(\CDOT)$ in (\ref{eq:P-phi})
% as the outer measure
Now we define the \emph{upper measure-theoretic probability} of $E$ as
\begin{equation}\label{eq:upper-measure}
  \UpProbMeas(E)
  :=
  \sup_{\phi}\Prob^{\phi}(E).
\end{equation}
% We will also need this definition for arbitrary sets $E$,
% in which case we will understand $\Prob_{\phi}$
% as the outer measure corresponding to the original $\Prob_{\phi}$.

\begin{remark}\label{rem:naive}
  Our definition (\ref{eq:upper-measure})
  is not fully adequate from the intuitive point of view:
  even if we are willing to assume that Forecaster follows some forecasting strategy
  (which is a non-trivial assumption: cf.\ the discussion in \cite{dawid:1985},
  pp.~1255--1256),
  why should this forecasting strategy depend only on the past outcomes?
  For example,
  a meteorologist forecasting rain might have data about temperatures, winds, etc.
  (See \cite{dawid:1985}, Section 9, for further discussion.)
  A more satisfactory definition would involve a supremum
  over all probability spaces equipped with a filtration
  and for each such a probability space
  a further supremum over all forecasting systems
  adapted to the corresponding filtration
  (with a natural more general definition of a forecasting system).
  Our definition (\ref{eq:upper-measure})
  is the simplest one mathematically
  and leads to the strongest inequality
  $\UpProb(E)\le\UpProbMeas(E)$
  (for the analytic sets),
  which is the non-trivial part of our main result, Theorem \ref{thm:main}.
\end{remark}

\section{Main result}
\label{sec:result}

Now we are have all ingredients needed to state our main result.

\begin{theorem}\label{thm:main}
  For all analytic sets $E\subseteq\Pi$,
  $\UpProb(E)=\UpProbMeas(E)$.
\end{theorem}

Intuitively, this theorem establishes the equivalence
between the purely prequential and Bayesian viewpoints
in the framework of probability forecasting.
The definition of measure-theoretic probability is Bayesian,
in that Forecaster is modeled as a coherent subjectivist Bayesian
having a joint probability distribution over the sequences of outcomes
(cf.\ \cite{dawid:1982}, Section 1);
we represent this joint probability distribution as a forecasting system.
Rejecting his forecasts is the same as rejecting all forecasting systems
that could have produced those forecasts:
cf.\ the $\sup_{\phi}$ in (\ref{eq:upper-measure}).
The definition of game-theoretic probability is purely prequential,
in that it does not postulate any joint probability distribution
behind the forecasts;
the latter are used for testing directly.

\ifFULL\bluebegin
  In an important sense,
  game-theoretic probability is dual to measure-theoretic probability;
  in particular,
  we have $\sup$ in the definition of the latter
  and $\inf$ in the definition of the former.
\blueend\fi

\begin{remark}
  As discussed in the previous section (Remark \ref{rem:naive}),
  our Bayesian forecaster is somewhat naive:
  he conditions only on the observed outcomes.
  It would be easy
  (but would complicate the exposition)
  to allow Reality to issue a \emph{signal} $s_n$,
  taking one of a finite number of values,
  before Forecaster chooses his forecast $p_n$.
  Allowing both farthingales and forecasting systems
  to depend on the signals,
  one could still prove that 
  $\UpProb(E)=\UpProbMeas(E)$
  for all analytic $E\subseteq\Pi$
  following the proof of Theorem \ref{thm:main}.
\end{remark}

In this section we will only prove the inequality $\ge$
in Theorem \ref{thm:main}.
It turns out that this inequality holds for all sets $E$,
not necessarily analytic.
\begin{theorem}\label{thm:easy-way}
  For all sets $E\subseteq\Pi$,
  $\UpProb(E)\ge\UpProbMeas(E)$.
\end{theorem}

The simple proof of Theorem \ref{thm:easy-way}
will follow from Ville's inequality
(\cite{ville:1939}, p.~100;
in modern probability textbooks
this result is often included among ``Doob's inequalities'':
see, e.g., \cite{shiryaev:1996}, Theorem VII.3.1.III).

Let $\phi$ be a forecasting system.
A \emph{martingale} w.r.\ to $\phi$
is a function $V:\Omega^{\diamond}\to(-\infty,\infty]$ satisfying
\begin{equation*}%\label{eq:martingale}
   V(x)
   =
   (1-\phi(x)) V(x,0) + \phi(x) V(x,1)
\end{equation*}
for all $x\in\Omega^{\diamond}$
(with the same convention $0\infty:=0$).
% If we replace ``$=$'' by ``$\ge$''
%% (respectively, by ``$\le$'')
% in (\ref{eq:martingale}),
% we get the definition of a \emph{supermartingale}
%% (respectively, \emph{submartingale})
% w.r.\ to $\phi$.
\begin{proposition}[\cite{ville:1939}]\label{prop:Ville}
  If $\phi$ is a forecasting system,
  $V$ is a non-negative martingale w.r.\ to $\phi$,
  and $C>0$,
  \begin{equation*}
    \Prob_{\phi}
    \left\{
      \omega\in\Omega
      \st
      \sup_n V(\omega^n) \ge C
    \right\}
    \le
    \frac{V(\Lambda)}{C}.
  \end{equation*}
\end{proposition}

If $V$ is a farthingale,
the function $V^{\phi}:\Omega^{\diamond}\to(-\infty,\infty]$
defined by
\begin{equation*}%\label{eq:V-phi}
  V^{\phi}(x)
  :=
  V\left(x^{\phi}\right),
  \quad
  x\in\Omega^{\diamond},
\end{equation*}
is a martingale w.r.\ to $\phi$.
It is important that this statement
does not require measurability of the farthingale $V$;
even if $V$ is not measurable,
$V^{\phi}$ is always measurable,
like any other function on $\Omega^{\diamond}$
(which is why there was no need to include the requirement of measurability
in our definition of a martingale).

\begin{proof}[Proof of Theorem \ref{thm:easy-way}]
  Let $E\subseteq\Pi$.
  It suffices to prove that $\Prob_{\phi}(E^{\phi})\le V(\Lambda)$
  for any forecasting system $\phi$
  and any non-negative farthingale $V$ satisfying
  $\limsup_n V(\pi^n)\ge1$ for all $\pi \in E$.
  Fix such $\phi$ and $V$.
  Then $V^{\phi}$ is a non-negative martingale w.r.\ to $\phi$ satisfying
  $\limsup_n V^{\phi}(\omega^n)\ge1$ for all $\omega \in E^{\phi}$.
  Applying Proposition \ref{prop:Ville} to $V^{\phi}$,
  we can see that indeed
  $\Prob_{\phi}(E^{\phi})\le V^{\phi}(\Lambda) = V(\Lambda)$.
\end{proof}

\section{Proof of the inequality $\le$ in Theorem \ref{thm:main}}
\label{sec:proof}

We start from proving a special case of Theorem \ref{thm:main}.
% A more general version of this result can be found
% in \cite{vovk/shen:2008journal}.)
\begin{lemma}\label{lem:compact}
  If $E\subseteq\Pi$ is a compact set,
  $\UpProbMeas(E)=\UpProb(E)$.
\end{lemma}
\begin{proof}
  Fix a compact prequential event $E\subseteq\Pi$.
  (Of course, ``compact'' is the same thing as ``closed'' in this context.)
  Represent $E$ as the intersection $E=\cap_{i=1}^{\infty}E_i$
  of a nested sequence $E_1\supseteq E_2\supseteq\cdots$
  of closed sets such that
  \begin{equation}\label{eq:level-i}
    \forall\pi\in\Pi:
    \pi\in E_i \Longrightarrow \Gamma(\pi^i)\subseteq E_i
  \end{equation}
  is satisfied for all $i$.
  Informally, $E_i$ is a property of the first $i$ forecasts and outcomes.
  % This set is not only closed but also open
  % when Forecaster's move space is finite.
  For each $i=1,2,\ldots$, define a superfarthingale $W_i$
  by setting
  \begin{equation}\label{eq:W-base}
    W_i(x)
    :=
    \begin{cases}
      1 & \text{if $\Gamma(x)\subseteq E_i$}\\
      0 & \text{otherwise}
    \end{cases}
  \end{equation}
  for all $x\in\Pi^{\ge i}$
  and then proceeding inductively as follows.
  If $W_i(x)$ is already defined for $x\in\Pi^n$, $n=i,i-1,\ldots,1$,
  define $W_i(x)$, for each $x\in\Pi^{n-1}$, by
  \begin{equation}\label{eq:sup}
    W_i(x)
    :=
    \sup_{p\in[0,1]}
    \bigl(
      (1-p) W_i(x,p,0)
      +
      p W_i(x,p,1)
    \bigr).
  \end{equation}
  It is clear that $W_1\ge W_2\ge\cdots$.

  Let us check that $W_i(x)$ is upper semicontinuous
  as a function of $x\in\Pi^{\diamond}$.
  By (\ref{eq:W-base}) this is true for $x\in\Pi^{\ge i}$.
  Suppose this is true for $x\in\Pi^n$, $n\in\{i,i-1,\ldots,2\}$,
  and let us prove that it is true for $x\in\Pi^{n-1}$,
  using the inductive definition (\ref{eq:sup}).
  It is clear that
  $
    f(x,p)
    :=
    (1-p) W_i(x,p,0)
    +
    p W_i(x,p,1)
  $
  is upper semicontinuous as function of $p\in[0,1]$ and $x\in\Pi^{n-1}$.
  It is well known that $\sup_p f(x,p)$ is upper semicontinuous
  whenever $f$ is upper semicontinuous
  and $x$ and $p$ range over compact sets
  (see, e.g., \cite{dellacherie:1972}, Theorem I.2(d)).
  \ifnotJOURNAL
    A simple proof of a slightly more general fact will be given below
    in Lemma \ref{lem:upper-semicontinuity}.
  \fi
  Therefore, $W_i(x)=\sup_{p\in[0,1]}f(x,p)$
  is an upper semicontinuous function of $x\in\Pi^{n-1}$.

  An important implication of the upper semicontinuity of $W_i$
  and the compactness of $[0,1]$
  is that the supremum in (\ref{eq:sup}) is attained:
  it is easy to check that an upper semicontinuous function
  attains its supremum over a compact set
  (cf.\ \cite{engelking:1989}, Problem 3.12.23(g)).
  For each $i=1,2,\ldots$,
  we can now define a forecasting system $\phi_i$ as follows.
  For each $x\in\Omega^n$, $n=0,1,\ldots,i-1$,
  choose $\phi_i(x)$ such that
  \begin{multline*}
    (1-\phi_i(x)) W_i(x^{\phi_i},\phi_i(x),0)
    +
    \phi_i(x) W_i(x^{\phi_i},\phi_i(x),1)\\
    =
    \sup_{p}
    \bigl(
      (1-p) W_i(x^{\phi_i},p,0)
      +
      p W_i(x^{\phi_i},p,1)
    \bigr)
    =
    W_i(x^{\phi_i})
  \end{multline*}
  (this is an inductive definition;
  in particular,
  $x^{\phi_i}$ is already defined at the time of defining $\phi_i(x)$).
  For $x\in\Omega^{\ge i}$, set, for example, $\phi_i(x):=0$.
  The important property of $\phi_i$
  is that $W_i^{\phi_i}$ is a martingale w.r.\ to $\phi_i$,
  and so $\Prob^{\phi_i}(E_i)=W_i(\Lambda)$.

  Since the set $\Phi$ of all forecasting systems is compact in the product topology,
  the sequence $\phi_i$ has a convergent subsequence $\phi_{i_k}$, $k=1,2,\ldots$;
  let $\phi:=\lim_{k\to\infty}\phi_{i_k}$.
  We assume, without loss of generality, $i_1<i_2<\cdots$.
  Set
  \begin{equation*}
    c
    :=
    \inf_i W_i(\Lambda)
    =
    \lim_{i\to\infty} W_i(\Lambda).
  \end{equation*}
  Fix an arbitrarily small $\epsilon>0$.
  Let us prove that $\Prob_{\phi}(E^{\phi})\ge c-\epsilon$.
  Let $K\in\bbbn$.
  The restriction of $\Prob_{\phi_{i_k}}$ to $\Omega^{i_K}$
  (more formally, the probability measure assigning weight
  $\Prob_{\phi_{i_k}}(\Gamma(x))$ to each singleton $\{x\}$,
  $x\in\Omega^{i_K}$)
  comes within $\epsilon$
  of the restriction of $\Prob_{\phi}$ to $\Omega^{i_K}$
  in total variation distance
  from some $k$ on;
  let the total variation distance be at most $\epsilon$
  for all $k\ge K'\ge K$.
  Let $k\ge K'$.
  Since $\Prob_{\phi_{i_k}}(E_{i_k}^{\phi_{i_k}})\ge c$,
  it is also true that $\Prob_{\phi_{i_k}}(E_{i_K}^{\phi_{i_k}})\ge c$;
  therefore, it is true that
  $\Prob_{\phi}(E_{i_{K}}^{\phi_{i_k}})\ge c-\epsilon$.
  By Fatou's lemma, we now obtain
  \begin{equation}\label{eq:Fatou}
    \Prob_{\phi}
    \left(
      \limsup_{k}
      E_{i_{K}}^{\phi_{i_k}}
    \right)
    \ge
    \limsup_{k\to\infty}
    \Prob_{\phi}(E_{i_{K}}^{\phi_{i_k}})
    \ge
    c-\epsilon.
  \end{equation}

  Let us check that
  \begin{equation}\label{eq:limsup}
    \limsup_{k}
    E_{i_{K}}^{\phi_{i_k}}
    \subseteq
    E_{i_{K}}^{\phi}.
  \end{equation}
  Indeed, let $\omega\notin E_{i_{K}}^{\phi}$,
  i.e., $\omega^{\phi}\notin E_{i_{K}}$.
  Since $\phi_{i_k}\to\phi$ in the product topology
  and the set $E_{i_K}$ is closed,
  $\omega^{\phi_{i_k}}\notin E_{i_K}$ from some $k$ on.
  This means that
  $\omega\in E_{i_K}^{\phi_{i_k}}$ for only finitely many $k$,
  i.e.,
  $
    \omega
    \notin
    \limsup_{k}
    E_{i_{K}}^{\phi_{i_k}}
  $.

  From (\ref{eq:Fatou}) and (\ref{eq:limsup})
  we can see that $\Prob_{\phi}(E_{i_K}^{\phi})\ge c-\epsilon$,
  for all $K\in\bbbn$.
  This implies $\Prob_{\phi}(E^{\phi})\ge c-\epsilon$.
  Since this holds for all $\epsilon$,
  $\Prob_{\phi}(E^{\phi})\ge c$.

  The rest of the proof is easy:
  since
  \begin{equation*}
    \UpProb(E)
    \le
    c
    \le
    \Prob_{\phi}(E^{\phi})
    \le
    \UpProbMeas(E)
    \le
    \UpProb(E)
  \end{equation*}
  (the last inequality following from Theorem~\ref{thm:easy-way}),
  we have
  \begin{equation*}
    \UpProb(E)
    =
    c
    =
    \Prob_{\phi}(E^{\phi})
    =
    \UpProbMeas(E).
    \qedhere
  \end{equation*}
\end{proof}

\ifnotJOURNAL
  In the proof of Lemma \ref{lem:compact}
  we referred to the following simple result.
  % in \cite{dellacherie:1972};
  % we will now state and prove a slightly more general result.
  \begin{lemma}\label{lem:upper-semicontinuity}
    Suppose $X$ and $Y$ are topological spaces
    and $Y$ is compact.
    If a function $f:X\times Y\to\bbbr$ is upper semicontinuous,
    then the function $x\in X\mapsto g(x):=\sup_{y\in Y} f(x,y)$
    is also upper semicontinuous.
  \end{lemma}
  \begin{proof}
    For any $c\in\bbbr$,
    we are required to show that the set $G:=\{x\st\sup_y f(x,y)<c\}$ is open.
    Let $x\in G$.
    For any $y\in Y$ there exists a neighborhood $O'_y$ of $x$
    and a neighborhood $O''_y$ of $y$
    such that, for some $\epsilon>0$,
    $f(x',y')<c-\epsilon$
    for all $x'\in O'_y$ and all $y'\in O''_y$.
    By the compactness of $Y$,
    there is a finite family $O''_{y_1},\ldots,O''_{y_K}$ that covers $Y$.
    The intersection of $O'_{y_1},\ldots,O'_{y_K}$
    will contain $x$ and will be a subset of $G$.
    Therefore, $G$ is indeed open.
    % and $g$ is upper semicontinuous

    The argument in \cite{dellacherie:1972}, proof of Theorem I.2(d),
    is even simpler,
    but it assumes that $X$ is compact
    (which is, however, sufficient for the purpose of Lemma \ref{lem:compact}).
  \end{proof}
\fi

The idea of the proof of Theorem \ref{thm:main}
is to extend Lemma \ref{lem:compact} to the analytic sets
using Choquet's capacitability theorem (stated below).
Remember that a function $\gamma$ (such as $\UpProb$ or $\UpProbMeas$)
mapping the power set of a topological space $X$ (such as $\Pi$) to $[0,\infty)$
is a \emph{capacity} if:
\begin{itemize}
\item
  for any subsets $A$ and $B$ of $X$,
  \begin{equation}\label{eq:condition-1}
    A\subseteq B
    \Longrightarrow
    \gamma(A)\le\gamma(B);
  \end{equation}
\item
  for any nested increasing sequence
  $A_1\subseteq A_2\subseteq\cdots$
  of arbitrary subsets of $X$,
  \begin{equation}\label{eq:condition-2}
    \gamma
    \left(
      \cup_{i=1}^{\infty} A_i
    \right)
    =
    \lim_{i\to\infty}
    \gamma(A_i);
  \end{equation}
\item
  for any nested decreasing sequence
  $K_1\supseteq K_2\supseteq\cdots$
  of compact sets in $X$,
  \begin{equation}\label{eq:condition-3}
    \gamma
    \left(
      \cap_{i=1}^{\infty} K_i
    \right)
    =
    \lim_{i\to\infty}
    \gamma(K_i).
  \end{equation}
\end{itemize}
Condition (\ref{eq:condition-3}) is sometimes replaced
by a different condition
which is equivalent to (\ref{eq:condition-3})
for compact metrizable spaces $X$:
cf.\ \cite{kechris:1995}, Definition 30.1.

It turns out that both $\UpProb$ and $\UpProbMeas$ are capacities.
We start from $\UpProb$.
% (this result is of a certain independent interest)
\begin{theorem}\label{thm:game-capacity}
  The set function $\UpProb$ is a capacity.
\end{theorem}
It is obvious that $\UpProb$ satisfies condition (\ref{eq:condition-1}).
The following two statements establish
conditions (\ref{eq:condition-2}) and (\ref{eq:condition-3}).
Condition (\ref{eq:condition-3}) is easier to check:
it can be extracted from the proof of Lemma \ref{lem:compact}.
\begin{lemma}\label{lem:condition-3}
  If $K_1\supseteq K_2\supseteq\cdots$ is a nested sequence of compact sets in $\Pi$,
  \begin{equation}\label{eq:condition-3-UpProb}
    \UpProb
    \left(
      \cap_{i=1}^{\infty} K_i
    \right)
    =
    \lim_{i\to\infty}
    \UpProb(K_i).
  \end{equation}
\end{lemma}
\begin{proof}
  We will use the equality
  $
    \UpProb(E)
    =
    \lim_{i\to\infty}
    \UpProb(E_i)
  $,
  in the notation of the proof of Lemma \ref{lem:compact}.
  This equality follows from
  \begin{equation*}
    \UpProb(E)
    =
    c
    =
    \lim_{i\to\infty}
    W_i(\Lambda)
    \ge
    \lim_{i\to\infty}
    \UpProb(E_i)
  \end{equation*}
  (the opposite inequality is obvious).

  Represent each $K_n$ in the form $K_n=\cap_{i=1}^{\infty}E_i$
  where $E_1\supseteq E_2\supseteq\cdots$
  and each $E_i$ satisfies (\ref{eq:level-i});
  we will write $K_{n,i}$ in place of $E_i$.
  Without loss of generality we will assume that
  $
    K_{1,i}\supseteq K_{2,i}\supseteq\cdots
  $
  for all $i$.
  Then the set $K:=\cap_{i=1}^{\infty} K_i$ can be represented as
  $K=\cap_{i=1}^{\infty}K_{i,i}$,
  and so (\ref{eq:condition-3-UpProb}) follows from
  \begin{multline*}
    \UpProb(K)
    =
    \UpProb
    \left(
      \cap_{i=1}^{\infty} K_{i,i}
    \right)
    =
    \lim_{i\to\infty}
    \UpProb(K_{i,i})
    =
    \lim_{n\to\infty}
    \lim_{i\to\infty}
    \UpProb(K_{n,i})\\
    =
    \lim_{n\to\infty}
    \UpProb
    \left(
      \cap_{i=1}^{\infty}
      K_{n,i}
    \right)
    =
    \lim_{n\to\infty}
    \UpProb(K_n).
    \qedhere
  \end{multline*}
\end{proof}

To check condition (\ref{eq:condition-2}) for $\UpProb$,
we will need the game-theoretic version,
proved in \cite{\ShaferLevy},
of L\'evy's zero-one law
(\cite{levy:1937}, Section 41).
% (pp.~128--130)
For each $x\in\Pi^{\diamond}$,
define the \emph{conditional upper game-theoretic probability} of $E\subseteq\Pi$ by
\begin{multline*}
  \UpProb(E\givn x)
  :=\\
  \inf
  \left\{
    a
    \st
    \exists V:
    V(x)=a
    \text{ and }
    \forall\pi\in E\cap\Gamma(x):
    \limsup_n V(\pi^n)\ge1
  \right\},
\end{multline*}
where $V$ ranges over the non-negative (super)farthingales.
% Similarly,
% for any $n\in\bbbn$
% and any sequence $x=(p_1,y_1,\ldots,p_{n-1},y_{n-1},p_n)$
% of alternating forecasts and outcomes,...
\begin{proposition}[\cite{\ShaferLevy}]\label{prop:levy}
  Let $E\subseteq\Pi$.
  For almost all $\pi\in E$,
  \begin{equation}\label{eq:goal}
    \UpProb(E\givn\pi^n)
    \to
    1
    % \text{ and }
    % \UpProb(E\givn\pi^{n-})
    % \to
    % 1
  \end{equation}
  as $n\to\infty$.
  (In other words,
  there exists a prequential event $N$ such that $\UpProb(N)=0$
  and (\ref{eq:goal}) holds for all $\pi\in E\setminus N$.)
\end{proposition}
\begin{proof}
  It suffices to construct a non-negative farthingale $V$ starting from 1
  that tends to $\infty$ on the sequences $\pi\in E$
  for which (\ref{eq:goal}) is not true.
  % We will only construct a non-negative farthingale starting from 1
  % that tends to $\infty$ on the sequences
  % for which
  % $
  %   \UpProb(E\givn\pi^n)
  %   \to
  %   1
  % $
  % is not true;
  % the case of
  % $
  %   \UpProb(E\givn\pi^{n-})
  %   \to
  %   1
  % $
  % is considered analogously.
  Without loss of generality
  we replace ``for which (\ref{eq:goal}) is not true'' by
  % we replace ``for which
  % $
  %   \UpProb(E\givn\pi^n)
  %   \to
  %   1
  % $
  % is not true'' by
  \begin{equation*}
    \liminf_{n\to\infty}
    \UpProb(E\givn\pi^n)
    <
    a,
  \end{equation*}
  where $a\in(0,1)$ is a given rational number
  (see Lemma \ref{lem:subadditivity}).
  % Choose any rational number $b\in(a,1)$.

  Let $\pi$ be any sequence in $\Pi$;
  we will define $V(\pi^n)$ by induction for $n=1,2,\ldots$
  (intuitively, we will describe a gambling strategy
  with capital process $V$).
  % imagining that we are moving along $\pi$:
  % first we define $V(\pi^n)$ for $n=0$, then we define $V(\pi^n)$ for $n=1$,
  % then for $n=2$, etc.
  Start with 1 monetary unit:
  $V(\Lambda):=1$.
  Keep setting $V(\pi^n):=1$, $n=1,2,\ldots$,
  until $\UpProb(E\givn\pi^n) < a$
  (if this never happens, $V(\pi^n)$ will be $1$ for all $n$).
  Let $N_1$ be the first $n$ when this happens:
  $\UpProb(E\givn\pi^{N_1}) < a$
  but $\UpProb(E\givn\pi^{n}) \ge a$ for all $n<N_1$.
  Choose a non-negative farthingale $S_1$ starting at $\pi^{N_1}$ from $1$,
  $S_1(\pi^{N_1}) = 1$,
  whose upper limit exceeds $1/a$ on all extensions of $\pi^{N_1}$ in $E$.
  Keep setting $V(\pi^n):=S_1(\pi^n)$, $n=N_1,N_1+1,\ldots$,
  until $S_1(\pi^n)$ reaches a value $s_1>1/a$.
  After that keep setting $V(\pi^n):=V(\pi^{n-1})$
  until $\UpProb(E\givn\pi^n) < a$.
  Let $N_2$ be the first $n$ when this happens.
  Choose a non-negative farthingale $S_2$ starting at $\pi^{N_2}$ from $s_1$,
  $S_2(\pi^{N_2}) = s_1$,
  whose upper limit exceeds $s_1/a$ on all extensions of $\pi^{N_2}$ in $E$.
  Keep setting $V(\pi^n):=S_2(\pi^n)$, $n=N_2,N_2+1,\ldots$,
  until $S_2(\pi^n)$ reaches a value $s_2>s_1(1/a)>(1/a)^2$.
  After that keep setting $V(\pi^n):=V(\pi^{n-1})$
  until $\UpProb(E\givn\pi^n) < a$.
  Let $N_3$ be the first $n$ when this happens.
  Choose a non-negative farthingale $S_3$ starting at $\pi^{N_3}$ from $s_2$
  whose upper limit exceeds $s_2/a$ on all extensions of $\pi^{N_3}$ in $E$.
  Keep setting $V(\pi^n):=S_3(\pi^n)$, $n=N_3,N_3+1,\ldots$,
  until $S_3$ reaches a value $s_3>s_2(1/a)>(1/a)^3$.
  And so on.
\end{proof}

\begin{lemma} % \label{lem:condition-2}
  If $A_1\subseteq A_2\subseteq\cdots\subseteq\Pi$
  is a nested sequence of prequential events,
  \begin{equation}\label{eq:condition-2-UpProb}
    \UpProb
    \left(
      \cup_{i=1}^{\infty} A_i
    \right)
    =
    \lim_{i\to\infty}
    \UpProb(A_i).
  \end{equation}
\end{lemma}
\begin{proof}
  Let $A_1,A_2,\ldots$ be a nested increasing sequence of prequential events.
  The non-trivial inequality in (\ref{eq:condition-2-UpProb}) is $\le$.
  For each $A_i$ the process
  \begin{equation*}
    S_i(x)
    :=
    \UpProb(A_i\givn x)
  \end{equation*}
  is a non-negative superfarthingale
  (see Lemma \ref{lem:UpProb-superfarthingale} below).
  By Proposition \ref{prop:levy},
  $\limsup_nS_i(\pi^n)\ge1$ for almost all $\pi\in A_i$.
  The sequence $S_i$ is increasing,
  $S_1\le S_2\le\cdots$,
  so the limit $S:=\lim_{i\to\infty}S_i=\sup_{i}S_i$
  exists and is a non-negative superfarthingale
  such that
  $
    S(\Lambda)
    =
    \lim_{i\to\infty}
    \UpProb(A_i)
  $
  and $\limsup_nS(\pi^n)\ge1$ for almost all $\pi\in\cup_iA_i$
  (by Lemma \ref{lem:subadditivity}).
  % (We have used the simple fact that a countable union
  % of prequential events of upper game-theoretic probability zero
  % has upper game-theoretic probability zero.)
  We can get rid of ``almost'' by adding to $S$
  a non-negative farthingale $V$ that starts at $V(\Lambda)<\epsilon$,
  for an arbitrarily small $\epsilon>0$,
  and satisfies $\limsup_n V(\pi^n)\ge1$
  for all $\pi\in\cup_iA_i$ violating $\limsup_nS(\pi^n)\ge1$.
\end{proof}

\begin{lemma}\label{lem:UpProb-superfarthingale}
  For any prequential event $E$,
  the function $x\in\Pi^{\diamond}\mapsto\UpProb(E\givn x)$
  is a superfarthingale.
\end{lemma}
\begin{proof}

  Suppose there are $x\in\Pi^{\diamond}$ and $p\in[0,1]$ such that
  \begin{equation*}
    \UpProb(E\givn x)
    <
    (1-p)
    \UpProb(E\givn x,p,0)
    +
    p
    \UpProb(E\givn x,p,1).
  \end{equation*}
  Then there exists a non-negative farthingale $V$
  with $\limsup_n V(\pi^n)\ge1$ for all $\pi\in E\cap\Gamma(x)$
  that satisfies
  \begin{equation*}
    V(x)
    <
    (1-p)
    \UpProb(E\givn x,p,0)
    +
    p
    \UpProb(E\givn x,p,1)
  \end{equation*}
  and, therefore,
  \begin{equation*}
    (1-p) V(x,p,0)
    +
    p V(x,p,1)
    <
    (1-p)
    \UpProb(E\givn x,p,0)
    +
    p
    \UpProb(E\givn x,p,1).
  \end{equation*}
  The last inequality implies that there exists $j\in\{0,1\}$ such that
  $V(x,p,j)<\UpProb(E\givn x,p,j)$,
  which is impossible.
\end{proof}

This completes the proof of Theorem \ref{thm:game-capacity}.
Let us now check that measure-theoretic probability is also a capacity.
% this fact seems to be less useful
% than its analogue for game-theoretic probability
\begin{lemma}\label{lem:meas-capacity}
  The set function $\UpProbMeas$ is a capacity.
\end{lemma}
\begin{proof}
  Property (\ref{eq:condition-1}) is obvious for $\UpProbMeas$.
  Property (\ref{eq:condition-3})
  follows from Lemmas \ref{lem:compact}
  and \ref{lem:condition-3}.

  Let us now check the remaining property (\ref{eq:condition-2}),
  with $\UpProbMeas$ as $\gamma$.
  Suppose there exists an increasing sequence
  $A_1\subseteq A_2\subseteq\cdots\subseteq X$ of prequential events
  such that
  \begin{equation*}
    \UpProbMeas
    \left(
      \cup_{i=1}^{\infty} A_i
    \right)
    >
    \lim_{i\to\infty}
    \UpProbMeas(A_i).
  \end{equation*}
  Let $\phi$ be a forecasting system satisfying
  \begin{equation*}
    \Prob^{\phi}
    \left(
      \cup_{i=1}^{\infty} A_i
    \right)
    >
    \lim_{i\to\infty}
    \UpProbMeas(A_i).
  \end{equation*}
  Then $\phi$ will satisfy
  %\begin{equation*}
  $
    \Prob^{\phi}
    \left(
      \cup_{i=1}^{\infty} A_i
    \right)
    >
    \lim_{i\to\infty}
    \Prob^{\phi}(A_i)
  $,
  %\end{equation*}
  which is equivalent to the obviously wrong
  $
    \Prob_{\phi}
    \left(
      \cup_{i=1}^{\infty} A_i^{\phi}
    \right)
    >
    \lim_{i\to\infty}
    \Prob_{\phi}(A_i^{\phi})
  $.
  \ifFULL\bluebegin

    Let us check that for any forecasting system $\phi$
    the set function $\Prob^{\phi}$ is a capacity.
    It is well known that $\Prob_{\phi}$ is a capacity
    (see, e.g., \cite{kechris:1995}, Exercise 30.3).
    The property
    $
      \Prob^{\phi}
      \left(
        \cap_{i=1}^{\infty} K_i
      \right)
      =
      \lim_{i\to\infty}
      \Prob^{\phi}(K_i)
    $,
    i.e.,
    $
      \Prob_{\phi}
      \left(
        \cap_{i=1}^{\infty} K_i^{\phi}
      \right)
      =
      \lim_{i\to\infty}
      \Prob_{\phi}(K_i^{\phi})
    $
    for compact $K_1\supseteq K_2\supseteq\cdots$
    follows from the corresponding property for $\Prob_{\phi}$
    and the fact that each $K_i^{\phi}$ is a compact set
    as the preimage $\{\omega\in\Omega\st\omega^{\phi}\in K_i\}$
    in the compact space $\Omega$
    of the closed set $K_i$
    under the continuous transformation $\omega\mapsto\omega^{\phi}$.
    The other two defining properties of capacities
    are even easier to check for $\Prob^{\phi}$.
  \blueend\fi
\end{proof}

In combination with Choquet's capacitability theorem,
Theorem \ref{thm:game-capacity} and Lemma \ref{lem:meas-capacity}
allow us to finish the proof of Theorem \ref{thm:main}.

\begin{ChoquetTheorem}[\cite{choquet:1954}]
  If $X$ is a compact metrizable space,
  $\gamma$ is a capacity on $X$,
  and $E\subseteq X$ is an analytic set,
  \begin{equation*}
    \gamma(E)
    =
    \sup
    \left\{
      \gamma(K)
      \st
      K \text{ is compact},
      K\subseteq E
    \right\}.
  \end{equation*}
\end{ChoquetTheorem}

For a proof of Choquet's theorem,
see, e.g., \cite{kechris:1995}, Theorem 30.13.

\begin{proof}[Proof of Theorem \ref{thm:main}]
  Combining Choquet's capacitability theorem
  (applied to the compact metrizable space $\Pi$),
  Lemma \ref{lem:compact},
  Theorem \ref{thm:game-capacity},
  and Lemma \ref{lem:meas-capacity},
  we obtain
  \begin{equation*}
    \UpProb(E)
    =
    \sup_{K\subseteq E}\UpProb(K)
    =
    \sup_{K\subseteq E}\UpProbMeas(K)
    =
    \UpProbMeas(E),
  \end{equation*}
  $K$ ranging over the compact sets.
\end{proof}

\begin{remark}
  The fact that game-theoretic probability and measure-theoretic probability
  are capacities
  has allowed us to prove their coincidence on the analytic sets,
  and it might be useful for other purposes as well.
  In general, neither of these capacities is \emph{strongly subadditive},
  in the sense of satisfying
  \begin{equation*}
    \gamma(A\cup B)
    +
    \gamma(A\cap B)
    \le
    \gamma(A)
    +
    \gamma(B)
  \end{equation*}
  for all prequential events $A$ and $B$.
  To demonstrate this it suffices,
  in view of Theorem \ref{thm:main},
  to find analytic sets $A$ and $B$ that violate
  \begin{equation}\label{eq:strictly-subadditive}
    \UpProb(A\cup B)
    +
    \UpProb(A\cap B)
    \le
    \UpProb(A)
    +
    \UpProb(B).
  \end{equation}
  We can define $\UpProb(E)$ for subsets of $\Pi^n$
  by (\ref{eq:upper-game}) with $\limsup_n$ omitted.
  This is an example of subsets $A$ and $B$ of $\Pi^2$
  for which (\ref{eq:strictly-subadditive}) is violated:
  \begin{align}
    A
    &=
    \left\{
      \left(
        0, 0, \frac12, 0
      \right),
      \left(
        \frac12, 0, 0, 0
      \right)
    \right\},\label{eq:A}\\
    B
    &=
    \left\{
      \left(
        0, 0, \frac12, 0
      \right),
      \left(
        \frac12, 1, 0, 0
      \right)
    \right\}.
    \label{eq:B}
  \end{align}
  For these subsets we have
  \begin{equation*}
    \UpProb(A\cup B)
    +
    \UpProb(A\cap B)
    =
    1 + \frac12
    >
    \frac12 + \frac12
    =
    \UpProb(A)
    +
    \UpProb(B).
  \end{equation*}
  To obtain an example of subsets $A$ and $B$ of the full prequential space $\Pi$
  for which (\ref{eq:strictly-subadditive}) is violated,
  it suffices to add $00\dots$ at the end of each element of the sets $A$ and $B$
  defined by (\ref{eq:A}) and (\ref{eq:B}).
\end{remark}

\section{Application to the limit theorems of probability theory}
\label{sec:application}

The \emph{lower game-theoretic probability} of a prequential event $E$
is defined to be $1-\UpProb(\Pi\setminus E)$.
Similarly, the \emph{lower measure-theoretic probability} of a prequential event $E$
is defined to be $1-\UpProbMeas(\Pi\setminus E)$.

The game-theoretic strong law of large numbers
(see, e.g., \cite{shafer/vovk:2001}, Section 3.3)
implies that (\ref{eq:calibration}) holds
% \begin{equation}\label{eq:SLLN}
%   \lim_{n\to\infty}
%   \frac1n
%   \sum_{i=1}^n
%   (y_i - p_i)
%   =
%   0
% \end{equation}
with lower game-theoretic probability one.
The standard martingale strong law of large numbers implies
that (\ref{eq:calibration}) holds with lower measure-theoretic probability one.
Our Theorem \ref{thm:main}
establishes the equivalence between these two statements.
Similarly,
Theorem \ref{thm:main} establishes the equivalence
between the game-theoretic law of the iterated logarithm for binary outcomes
(a special case of Theorems 5.1 and 5.2 in \cite{shafer/vovk:2001})
and the martingale law of the iterated logarithm for binary outcomes
in measure-theoretic probability theory.

Transition from game-theoretic to measure-theoretic laws of probability,
corresponding to the inequality $\ge$ in Theorem \ref{thm:main},
depends only on Ville's inequality,
and so can be easily done for a wide variety of prediction protocols
(see, e.g., \cite{shafer/vovk:2001}, Section 8.1).
Transition in the opposite direction,
corresponding to the inequality $\le$,
is more difficult,
and its feasibility has been demonstrated only in a very limited number of cases.

In an important respect Theorem \ref{thm:main}
is only an existence result.
For example,
in combination with the standard martingale strong law of large numbers
in measure-theoretic probability theory
it implies the game-theoretic strong law of large numbers
for binary outcomes,
but the resulting farthingale is very complex.
The corresponding strategy for the gambler
(or Skeptic, in the terminology of \cite{shafer/vovk:2001})
is also very complex.
This contrasts with the simple and efficient gambling strategies
designed in game-theoretic probability:
see, e.g., \cite{shafer/vovk:2001}, Section 3.2,
and \cite{kumon/takemura:2008}.

It would be interesting to design efficient general procedures
producing simple gambling strategies
witnessing that $\UpProb(E)=0$
for natural classes of prequential events satisfying $\UpProbMeas(E)=0$.
For example,
such a procedure might be applicable
to all prequential events satisfying $\UpProbMeas(E)=0$
and situated at a given low level of the Borel
% (in future, projective)
hierarchy.
This would allow an automatic procedure
of transition from measure-theoretic to constructive game-theoretic laws of probability:
e.g., the set of sequences (\ref{eq:sequence})
violating the strong law of large numbers (\ref{eq:calibration})
is in the class $\Sigma^0_3$ of the Borel hierarchy,
and the set of sequences violating the law of the iterated logarithm
is in $\Delta^0_4$.

In this article we have only considered the case
where the outcomes $y_n$ are restricted to the binary outcome space
$Y:=\{0,1\}$.
It is easy to extend our results to the case
where $Y$ is any finite set
and Forecaster outputs probability measures on $Y$,
interpreted as his probability forecasts for $y_n$.
It remains an open problem whether it is possible to modify
our definitions in a natural way
so that the equivalence
between game-theoretic and measure-theoretic probability
extends to a wide classes of outcome spaces and prequential events;
this would require imposing suitable measurability or topological conditions
on the farthingales (or superfarthingales)
used in the definition (\ref{eq:upper-game}) of game-theoretic probability.

\ifJOURNAL
  \section*{Acknowledgements}
\fi
\ifnotJOURNAL
  \subsection*{Acknowledgements}
\fi

This article has greatly benefitted from conversations
with Philip Dawid, Glenn Shafer, and Alexander Shen.
Its main result answers, within the prequential framework, a question
that has been asked independently by several people,
including Shafer and, more recently, Shen.
% Shen's version, where the forecast spaces $F_n$ are finite
% (his framework is described in \cite{chernov/etal:2008}),
% is especially interesting;
% I am grateful to him for reviving my interest in this problem.
I am grateful to Alexander Kechris
for his advice about capacities.
\ifnotJOURNAL
  This work was supported in part by EPSRC (grant EP/F002998/1).
\fi

\end{document}